\documentclass[11pt,reqno,a4paper]{amsart}

\usepackage[T1]{fontenc}
\usepackage[utf8]{inputenc}

\usepackage{libertinus}
\usepackage{microtype}
\usepackage{geometry}
\usepackage{amsmath,amssymb,amsthm,amsfonts}
\usepackage{mathtools}
\usepackage{mathrsfs}
\usepackage{bbm}

\numberwithin{equation}{section}

\usepackage{graphicx}
\usepackage{xcolor}
\usepackage{enumitem}
\setlist{itemsep=2pt,topsep=4pt}

\usepackage[
  colorlinks=true,
  linkcolor=blue,
  citecolor=blue,
  urlcolor=blue,
  pagebackref=true
]{hyperref}

\renewcommand*{\backrefalt}[4]{%
  \ifcase #1\relax
  \or
    {\footnotesize Cited on page #2.}%
  \else
    {\footnotesize Cited on pages #2.}%
  \fi
}

\usepackage[nameinlink,noabbrev]{cleveref}
\crefname{section}{Section}{Sections}
\Crefname{section}{Section}{Sections}
\crefname{subsection}{Subsection}{Subsections}
\Crefname{subsection}{Subsection}{Subsections}

\usepackage{titlesec}

\titleformat{\section}
  {\normalfont\large\bfseries\centering}
  {\thesection.}
  {0.75em}
  {}

\titlespacing*{\section}
  {0pt}
  {3.2ex plus 1ex minus .2ex}
  {1.6ex plus .3ex}

\titleformat{\subsection}
  {\normalfont\normalsize\bfseries}
  {\thesubsection.}
  {0.75em}
  {}

\titlespacing*{\subsection}
  {0pt}
  {2.4ex plus .8ex minus .2ex}
  {1.0ex plus .2ex}

\titleformat{\subsubsection}
  {\normalfont\normalsize\bfseries\itshape}
  {\thesubsubsection.}
  {0.75em}
  {}

\titlespacing*{\subsubsection}
  {0pt}
  {2.0ex plus .6ex minus .2ex}
  {0.7ex plus .2ex}

\usepackage{aliascnt}
\usepackage{etoolbox}

\newtheoremstyle{compact}
  {6pt}
  {1pt}
  {\itshape}
  {}
  {\bfseries}
  {.}
  {0.5em}
  {}

\theoremstyle{compact}
\newtheorem{theorem}{Theorem}[section]

\newaliascnt{proposition}{theorem}
\newtheorem{proposition}[proposition]{Proposition}
\aliascntresetthe{proposition}

\newaliascnt{lemma}{theorem}
\newtheorem{lemma}[lemma]{Lemma}
\aliascntresetthe{lemma}

\newaliascnt{corollary}{theorem}
\newtheorem{corollary}[corollary]{Corollary}
\aliascntresetthe{corollary}

\newtheorem{maintheorem}{Theorem}

\theoremstyle{definition}
\newaliascnt{definition}{theorem}

\aliascntresetthe{definition}

\newtheorem*{example}{Example}
\newtheorem*{fouriercondition}{Fourier condition \textup{(F)}}

\theoremstyle{definition}
\newtheorem{remark}[theorem]{Remark}

\AtBeginEnvironment{proof}{\vspace{6pt}}

\makeatletter
\AtBeginEnvironment{thebibliography}{%
  \def\@mklab#1{#1\hfil}%
}
\makeatother

\newcommand{\bR}{\mathbb{R}}
\newcommand{\bC}{\mathbb{C}}

\newcommand{\E}{\mathbb{E}}
\newcommand{\Pp}{\mathbb{P}}
\newcommand{\Var}{\operatorname{Var}}

\newcommand{\ind}{\mathbbm{1}}

\newcommand{\dto}{\xrightarrow{\,d\,}}

\begin{document}

\title{Randomized Spectral Inference for Hyperuniformity}

\author{Michael Bj\"orklund}
\address{Department of Mathematics, Chalmers University of Technology and University of Gothenburg, Gothenburg, Sweden}
\email{micbjo@chalmers.se}


\begin{abstract}
We test hyperuniformity from one large realization of a stationary point
process. Hyperuniformity is equivalent to the average $A_r$ of the structure
factor over $B_r$ vanishing as $r\downarrow0$, so the target is a low-frequency
average rather than the value of the structure factor at the origin, which need
not exist. We estimate $A_r$ from squared Fourier coefficients at frequencies
drawn uniformly from $B_r$; if these coefficients are asymptotically Gaussian
at almost every fixed frequency, the squared coefficients converge to
independent variables with mean $A_r$. Assuming
$A_r=s+c r^\alpha+\varepsilon_r$ with
$|\varepsilon_r|\leq Lr^\beta$ and given exponents $0<\alpha<\beta$, a
two-radius extrapolation removes the $r^\alpha$ term and estimates $s$ with
deterministic error of order $r^\beta$. This gives confidence bounds and a
one-sided test of $H_0:s=0$ with asymptotic level at most $\gamma$, consistent
against every fixed $s>0$ as $R\to\infty$, then the number of sampled
frequencies tends to infinity, and then $r\downarrow0$. Essentially free
translation actions with completely positive entropy satisfy the Fourier
assumption.
\end{abstract}

\maketitle

\section{Introduction}
\label{sec:introduction}

Let $\eta$ be a stationary point process on $\bR^d$ with positive intensity
$\rho_\eta$ and local second moments, observed in $B_R:=B(0,R)$. It is
hyperuniform if
\[
    \frac{\Var(\eta(B_R))}{\lambda_d(B_R)}\longrightarrow0
    \qquad (R\to\infty),
\]
where $\lambda_d$ denotes Lebesgue measure. Hyperuniformity \cite{TS03}, the
suppression of large-scale density fluctuations, now appears across statistical
physics, materials science and stochastic geometry; see \cite{Tor18}. Recent
statistical work includes \cite{HGBL23,KLH26,MBL26}. We ask how to test
hyperuniformity from a single observed pattern.

Write
\[
    \hat f(\xi):=\int_{\bR^d}f(x)e^{-2\pi i\langle x,\xi\rangle}\,d\lambda_d(x),
    \qquad
    \chi_\xi(x):=e^{2\pi i\langle x,\xi\rangle}.
\]
For compactly supported $f$, put
\[
    M(f):=\sum_{x\in\eta}f(x)-\rho_\eta\int_{\bR^d}f\,d\lambda_d.
\]
The identity
\[
    \E\bigl[M(f)\overline{M(g)}\bigr]
    =\int_{\bR^d}\hat f(\xi)\overline{\hat g(\xi)}\,d\sigma_\eta(\xi)
\]
for Schwartz $f,g$ defines the Bartlett measure $\sigma_\eta$. If
$\sigma_\eta\ll\lambda_d$, write
$d\sigma_\eta=\rho_\eta S_\eta\,d\lambda_d$, call $S_\eta$ the structure
factor, and fix a nonnegative Borel representative. For $r>0$, define
\[
    A_r:=\frac{\sigma_\eta(B_r)}{\rho_\eta\lambda_d(B_r)}
        =\frac1{\lambda_d(B_r)}\int_{B_r}S_\eta(\xi)\,d\lambda_d(\xi).
\]
Hyperuniformity is equivalent to $A_r\to0$ as $r\downarrow0$
\cite[Proposition~3.3]{BH24}. Moreover,
\[
    \frac{\Var(\eta(B_R))}{\lambda_d(B_R)}
    =\rho_\eta\int_{\bR^d}F_R(\xi)S_\eta(\xi)\,d\lambda_d(\xi),
    \qquad
    F_R:=\frac{|\widehat{\ind_{B_R}}|^2}{\lambda_d(B_R)},
\]
where $F_R$ has integral one and concentrates at the origin. This explains why
low-frequency mass governs number variance, while the intrinsic target remains
$A_r$: the density $S_\eta$ is defined only almost everywhere, so a value
$S_\eta(0)$ need not be meaningful.

At a frequency $U$ drawn uniformly from $B_r$, the squared modulus of a
smoothly weighted Fourier sum behaves for large $R$ like $S_\eta(U)$ times a
unit-mean exponential variable; averaging over many draws therefore estimates
$A_r$. We assume that $A_r$ has a leading term $c r^\alpha$ with a remainder
of order $r^\beta$, and combine the estimates at $r$ and $\kappa r$ to cancel
$c$. The remaining bias has an explicit bound, which we add to the normal
critical value. All Fourier values come from the same pattern and only the
frequencies are sampled independently; condition \textup{(F)} below gives the
asymptotic Gaussian independence at distinct fixed frequencies.

\medskip\noindent\textbf{Informal result.}\ 
Under condition \textup{(F)} and the low-frequency model below, this procedure
has asymptotic level at most $\gamma$ at each fixed radius and becomes
consistent against every fixed $s>0$ as the radius tends to zero.

Fix $h\in C_c^\infty(B_1;\bC)$ with
$\int|h|^2\,d\lambda_d=1$, and put
\[
    Z_R(\xi):=R^{-d/2}M\bigl(h(\cdot/R)\chi_\xi\bigr).
\]

\begin{fouriercondition}
For every $n\geq1$ there is a Borel set
$\Omega_\eta^{(n)}\subset(\bR^d)^n$ of full Lebesgue measure, contained in
\[
    \{(u_1,\ldots,u_n):u_i\neq0,\ u_i\neq\pm u_j\text{ for }i\neq j\},
\]
such that for every $(u_1,\ldots,u_n)\in\Omega_\eta^{(n)}$,
\[
    (Z_R(u_1),\ldots,Z_R(u_n))
    \dto (G_{u_1},\ldots,G_{u_n}),
\]
where the $G_{u_j}$ are independent centered circularly symmetric complex
Gaussian variables (their real and imaginary parts are independent centered
Gaussians of equal variance) satisfying
\[
    \E|G_{u_j}|^2=\rho_\eta S_\eta(u_j),
    \qquad \E G_{u_j}^2=0.
\]
\end{fouriercondition}

We link positive radii to the origin through the model
\begin{equation}
\label{eq:intro-integrated-model}
    A_r=s+c r^\alpha+\varepsilon_r,
    \qquad |\varepsilon_r|\leq Lr^\beta,
    \qquad 0<r\leq r_0,
\end{equation}
where $0<\alpha<\beta$ and $L,r_0>0$ are given, while $s\geq0$ and
$c\in\bR$ are unknown. When the remainder vanishes identically, any $L>0$ is
admissible. Here $s=\lim_{r\downarrow0}A_r$ is an averaged low-frequency
limit; no pointwise limit of $S_\eta$ at the origin is assumed.
We test
\[
    H_0:s=0
    \qquad\text{against}\qquad
    H_1:s>0.
\]
Some assumption across radii is unavoidable for the label-free procedure used
here: Proposition~\ref{prop:finite-radius-nonidentifiability} shows that different
low-frequency profiles can give the same sampled values at finitely many positive
radii once the frequency labels are discarded.

For $\kappa>1$ and $0<r\leq r_0/\kappa$, set
\[
    D_{\alpha,\beta}(\kappa)
    :=\frac{\kappa^\alpha+\kappa^\beta}{\kappa^\alpha-1},
    \qquad
    s_r:=\frac{\kappa^\alpha A_r-A_{\kappa r}}{\kappa^\alpha-1},
    \qquad
    b_r:=L D_{\alpha,\beta}(\kappa)r^\beta.
\]
Thus $s_r$ is the extrapolated intercept obtained from the values at $r$ and
$\kappa r$ when $r^\alpha$ is used as the scale. Equation
\eqref{eq:intro-integrated-model} gives
\[
    s_r-s
    =\frac{\kappa^\alpha\varepsilon_r-\varepsilon_{\kappa r}}
            {\kappa^\alpha-1},
    \qquad
    |s_r-s|\leq b_r.
\]

\subsection{The two-radius rule and main result}
\label{subsec:intro-main}

For an arbitrary pair of arrays
$y=(y_{a,j}:a\in\{r,\kappa r\},\,1\leq j\leq m)$ with $m\geq2$, let
\[
    \bar y_a:=\frac1m\sum_{j=1}^m y_{a,j},
    \qquad
    v_a^2(y):=\frac1{m-1}\sum_{j=1}^m(y_{a,j}-\bar y_a)^2,
\]
and define
\begin{equation}
\label{eq:generic-two-radius}
\begin{aligned}
    \widehat s_r(y)
    &:=\frac{\kappa^\alpha\bar y_r-\bar y_{\kappa r}}
            {\kappa^\alpha-1},\\
    \operatorname{se}_r^2(y)
    &:=\frac{\kappa^{2\alpha}v_r^2(y)+v_{\kappa r}^2(y)}
            {m(\kappa^\alpha-1)^2}.
\end{aligned}
\end{equation}
For $\gamma\in(0,1)$, let
\begin{equation}
\label{eq:generic-reject}
    \operatorname{Rej}_r(y;\gamma)
    :=\{\widehat s_r(y)>b_r+z_{1-\gamma}\operatorname{se}_r(y)\},
\end{equation}
where $z_{1-\gamma}$ is the $(1-\gamma)$-quantile of $N(0,1)$. This rule
tests the enlarged null $s_r\leq b_r$, which contains $H_0$ because
$|s_r-s|\leq b_r$; the allowance $b_r$ covers the unknown remainder at the
price of conservativeness when $s_r<b_r$.

For the fixed test function $h$, define the windowed periodogram
\begin{equation}
\label{eq:windowed-periodogram}
    I_{R,h}(u)
    :=\widehat\rho_R^{-1}R^{-d}
      \left|\sum_{x\in\eta\cap B_R}h(x/R)\chi_u(x)\right|^2,
    \qquad
    \widehat\rho_R:=\frac{\eta(B_R)}{\lambda_d(B_R)},
\end{equation}
on $\{\widehat\rho_R>0\}$, and set it equal to zero otherwise. The sum in
\eqref{eq:windowed-periodogram} is not centered, but its centering term
$\rho_\eta R^{d/2}\hat h(-Ru)$ vanishes for every fixed $u\neq0$ because
$\hat h$ is Schwartz; see Theorem~\ref{thm:observed-randomized-limit}.

Draw $m$ frequencies uniformly from each of $B_r$ and $B_{\kappa r}$,
independently of $\eta$, and set
\[
    Y^{(R)}_{a,j}:=I_{R,h}(U_{a,j}),
    \qquad a\in\{r,\kappa r\},\quad 1\leq j\leq m.
\]

\begin{maintheorem}[Hyperuniformity test]
\label{thm:main}
Assume that $\eta$ is stationary and ergodic, has positive intensity and local
second moments, $\sigma_\eta\ll\lambda_d$, condition \textup{(F)} and
\eqref{eq:intro-integrated-model} hold, and
\[
    \int_{B_{r_0}}S_\eta(\xi)^2\,d\lambda_d(\xi)<\infty.
\]
Fix $\gamma\in(0,1)$ and $\kappa>1$, and write
\[
    \pi_{R,m}(r):=\Pp\bigl(\operatorname{Rej}_r(Y^{(R)};\gamma)\bigr).
\]
For every fixed $0<r\leq r_0/\kappa$,
\begin{enumerate}
\item under $H_0$,
\[
    \limsup_{m\to\infty}\limsup_{R\to\infty}\pi_{R,m}(r)\leq\gamma;
\]
\item if $s_r>b_r$, in particular if $s>2b_r$,
\[
    \liminf_{m\to\infty}\liminf_{R\to\infty}\pi_{R,m}(r)=1.
\]
\end{enumerate}
Consequently, for every sequence $r_n\downarrow0$ with $\kappa r_n\leq r_0$,
\[
    \limsup_{n\to\infty}\limsup_{m\to\infty}\limsup_{R\to\infty}
    \pi_{R,m}(r_n)\leq\gamma
\]
under $H_0$, while for every fixed $s>0$,
\[
    \lim_{n\to\infty}\liminf_{m\to\infty}\liminf_{R\to\infty}
    \pi_{R,m}(r_n)=1.
\]
\end{maintheorem}

\medskip\noindent\textbf{Scope.}\ 
The guarantee holds within \eqref{eq:intro-integrated-model}. The exponents
$\alpha<\beta$, the constants $L,r_0$, and the choices $r,\kappa,m,h$ are
given, not estimated. Processes whose low-frequency average diverges or has
logarithmic corrections lie outside this model, and the theorem gives no
finite-$R$ calibration or rule for choosing $r$ and $m$ as functions of $R$.
Subsection~\ref{subsec:L-sensitivity} reports the largest $L$ for which the
observed data still lead to rejection.

\subsection{A concrete class satisfying \textup{(F)}}
\label{subsec:intro-cpe}

Condition \textup{(F)} is a fluctuation assumption at fixed frequencies,
whereas \eqref{eq:intro-integrated-model} controls low-frequency averages; the
two assumptions are independent. Completely positive entropy supplies a
concrete class for which \textup{(F)} is known. A translation action is
essentially free if almost every pattern has no nonzero period, and has
completely positive entropy (CPE) if every nontrivial equivariant factor has
positive entropy. Bernoulli actions, hence the Poisson process, are CPE; a
randomly shifted lattice is not, since it has periods and zero entropy.

The next result follows from \cite[Section~2.5, Theorem~A, Theorem~5.1 and
Corollary~5.5]{Bjo26}.

\begin{theorem}[CPE supplies \textup{(F)}]
\label{thm:cpe-supplies-F}
Let $\eta$ be a stationary point process on $\bR^d$ with positive intensity and
local second moments. If its translation action is essentially free and has
completely positive entropy, then the action is ergodic,
$\sigma_\eta\ll\lambda_d$, and condition \textup{(F)} holds. In fact, the
joint Fourier limit holds for every fixed finite family of smooth test functions.
\end{theorem}

CPE also gives ergodicity because the invariant $\sigma$-algebra is a
zero-entropy factor. Absolute continuity, even with a bounded continuous
density, does not by itself imply the corresponding Fourier CLT for all smooth
test functions; see \cite[Theorem~B]{Bjo26}. Section~\ref{subsec:examples} gives
examples on both sides of the test: the sine process is hyperuniform, whereas
the Poisson, thinned sine and Gaussian determinantal processes are not.

\subsection{Relation to existing work}
\label{subsec:intro-related}

The closest testing procedure is due to Klatt, Last and Henze \cite{KLH26}.
Their Fourier central limit theorem uses summability and integrability
conditions on higher-order factorial cumulants, including Brillinger mixing
\cite[Theorem~4.1]{KLH26}, and their statistical model is pointwise at selected
wave vectors. The approaches differ in three respects:
\begin{enumerate}
\item \emph{Dependence.} Condition \textup{(F)} can hold without their
cumulant assumptions; Theorem~\ref{thm:cpe-supplies-F} gives one such route.
\item \emph{Model.} They assume a pointwise expansion at selected wave vectors,
whereas we constrain only the averages $A_r$; Proposition~\ref{prop:pointwise-to-integrated}
shows that a pointwise expansion with a controlled remainder implies our model.
After randomization the limit law is generally a mixture of exponentials with
random mean $S_\eta(U)$, so their two-parameter exponential likelihood is not
available.
\item \emph{Calibration.} They calibrate a likelihood-ratio statistic
numerically, whereas we use a sample-variance normalization with standard
normal critical values.
\end{enumerate}
Mastrilli, B{\l}aszczyszyn and Lavancier \cite{MBL26} estimate the
hyperuniformity exponent and obtain asymptotic normality under Brillinger
mixing when $\alpha<d$. Hawat, Gautier, Bardenet and Lachi\`eze-Rey
\cite{HGBL23} and Mastrilli \cite{Mas25} study structure-factor estimation and
related inference. Random-frequency Fourier limits have a time-series
precedent in Peligrad and Wu \cite{PW10}.

\section{Randomized windowed periodograms}
\label{sec:experiment}

This section shows that the windowed periodogram at random frequencies
converges to the variables $S_\eta(U)W$ used below.

\begin{theorem}[Randomized windowed-periodogram limit]
\label{thm:observed-randomized-limit}
Assume that $\eta$ is stationary and ergodic, has positive intensity and local
second moments, $\sigma_\eta\ll\lambda_d$, and condition \textup{(F)} holds. Let $(U_1,\ldots,U_m)$ have any joint law absolutely
continuous with respect to Lebesgue measure on $(\bR^d)^m$ and be independent
of $\eta$. For fixed $m$,
\[
    (U_j,I_{R,h}(U_j))_{j=1}^m
    \dto
    (U_j,S_\eta(U_j)W_j)_{j=1}^m,
\]
where $W_1,\ldots,W_m$ are independent $\operatorname{Exp}(1)$ variables,
independent of the frequencies.
\end{theorem}

\begin{proof}
The pointwise ergodic theorem for balls gives
$\widehat\rho_R\to\rho_\eta$ almost surely and hence
$\Pp(\widehat\rho_R=0)\to0$; see also \cite[Section~6.2]{Bjo26}. Since
$\operatorname{supp}h\subset B_1$, the sum in
\eqref{eq:windowed-periodogram} equals the corresponding sum over all of
$\eta$, and for fixed $u\neq0$,
\[
    R^{-d/2}\sum_{x\in\eta}h(x/R)\chi_u(x)
    =Z_R(u)+\rho_\eta R^{d/2}\hat h(-Ru),
\]
with the second term tending to zero because $\hat h$ is Schwartz.

Let $\mu_U$ be the joint law of the frequency tuple. Absolute continuity gives
$\mu_U(\Omega_\eta^{(m)})=1$. Conditional on $U=u\in\Omega_\eta^{(m)}$,
condition \textup{(F)}, Slutsky's theorem and the continuous mapping theorem
give
\[
    (I_{R,h}(u_j))_{j=1}^m
    \dto
    (S_\eta(u_j)|G_j|^2)_{j=1}^m,
\]
where the $G_j$ are independent standard complex Gaussians with
$\E|G_j|^2=1$. Since $|G_j|^2\sim\operatorname{Exp}(1)$, integrating bounded
continuous test functions against $\mu_U$ proves the claim.
\end{proof}

\begin{corollary}[Two-radius randomized limit]
\label{cor:two-radius-limit}
Under the assumptions of Theorem~\ref{thm:observed-randomized-limit}, for
$a\in\{r,\kappa r\}$, let $U_{a,1},\ldots,U_{a,m}$ be independent and
uniform on $B_a$, and let $W_{a,j}$ be independent $\operatorname{Exp}(1)$
variables, independent of the frequencies. Define
\[
    Y^{(\infty)}_{a,j}:=S_\eta(U_{a,j})W_{a,j}.
\]
Then the observed array from Subsection~\ref{subsec:intro-main} satisfies
\[
    Y^{(R)}\dto Y^{(\infty)}
    \qquad (R\to\infty)
\]
for every fixed $m,r$ and $\kappa$.
\end{corollary}

\begin{proof}
Apply Theorem~\ref{thm:observed-randomized-limit} to the full tuple of the
$2m$ sampled frequencies.
\end{proof}

\section{Two-radius inference}
\label{sec:low-frequency}

\subsection{Why a low-frequency assumption is needed}
\label{subsec:finite-radius-identifiability}

For nonnegative $S\in L^1_{\rm loc}(\lambda_d)$, write
\[
    A_r[S]:=\frac1{\lambda_d(B_r)}\int_{B_r}S(\xi)\,d\lambda_d(\xi).
\]
By the next proposition, a function vanishing near the origin and one that does
not can generate identical sampled values at finitely many positive radii. The
test discards the frequency labels, which yields i.i.d. variables with mean
$A_r$; using the labels would require a pointwise model for $S_\eta$, which we
avoid.

\begin{proposition}[Nonidentifiability from finitely many positive radii]
\label{prop:finite-radius-nonidentifiability}
Fix $0<r_1<\cdots<r_J$. There exist bounded, compactly supported, radial
functions $S_0,S_1\geq0$ such that
\[
    A_{r_j}[S_0]=A_{r_j}[S_1],
    \qquad j=1,\ldots,J,
\]
but
\[
    \lim_{r\downarrow0}A_r[S_0]=0,
    \qquad
    \lim_{r\downarrow0}A_r[S_1]>0.
\]
More strongly, for arbitrary finite sample sizes $m_1,\ldots,m_J$, let
$U_{jk}\sim\operatorname{Unif}(B_{r_j})$ and
$W_{jk}\sim\operatorname{Exp}(1)$ be mutually independent. Then
\begin{equation}
\label{eq:finite-radius-same-law}
    \bigl(S_0(U_{jk})W_{jk}\bigr)_{j,k}
    \stackrel{d}=
    \bigl(S_1(U_{jk})W_{jk}\bigr)_{j,k}.
\end{equation}
\end{proposition}

\begin{proof}
Choose $\delta>0$ so small that $\Delta:=2^{1/d}\delta<r_1$, fix
$\vartheta>0$, and set
\[
    S_0=\vartheta\ind_{B_\Delta\setminus B_\delta},
    \qquad
    S_1=\vartheta\ind_{B_\delta}.
\]
Since $\lambda_d(B_\Delta\setminus B_\delta)=\lambda_d(B_\delta)$, for every
$j,k$ both $S_0(U_{jk})$ and $S_1(U_{jk})$ equal $\vartheta$ with probability
$\lambda_d(B_\delta)/\lambda_d(B_{r_j})$ and equal zero otherwise. The arrays
$(S_i(U_{jk}))_{j,k}$ have the same joint law, and multiplication by the
independent exponential variables gives \eqref{eq:finite-radius-same-law}.
For $0<r<\delta$, one has $A_r[S_0]=0$ and $A_r[S_1]=\vartheta$.
\end{proof}

The functions $S_0,S_1$ need not be point-process structure factors; the
proposition concerns only the reduction to sampled values.

\subsection{The low-frequency assumption and examples}
\label{subsec:examples}

The examples use
$S_\eta=1+\rho_\eta\widehat{(g-1)}$, where $g$ is the pair-correlation
function whenever this formula is defined. A pointwise expansion implies
\eqref{eq:intro-integrated-model} but is not required.

\begin{proposition}[Pointwise sufficient condition]
\label{prop:pointwise-to-integrated}
If a representative of $S_\eta$ satisfies
\[
    |S_\eta(\xi)-s-t\|\xi\|^\alpha|\leq C\|\xi\|^\beta,
    \qquad \|\xi\|\leq r_0,
\]
then \eqref{eq:intro-integrated-model} holds with
\[
    c=\frac{d}{d+\alpha}t,
    \qquad
    L=\frac{d}{d+\beta}C.
\]
\end{proposition}

\begin{proof}
For $\theta>0$,
\[
    \frac1{\lambda_d(B_r)}\int_{B_r}\|\xi\|^\theta\,d\lambda_d(\xi)
    =\frac{d}{d+\theta}r^\theta.
\]
Averaging the pointwise bound proves the claim.
\end{proof}

\begin{example}[Poisson process]
For a homogeneous Poisson process, $S_\eta\equiv1$, so $s=1$, $c=0$ and the
remainder vanishes. Its translation action is essentially free and Bernoulli,
so Theorem~\ref{thm:cpe-supplies-F} applies. Here the limiting periodogram value
is simply $\operatorname{Exp}(1)$.
\end{example}

\begin{example}[Sine process]
For the unit-intensity sine process on $\bR$,
$K_{\rm sine}(x)=\sin(\pi x)/(\pi x)$ and
$\widehat K_{\rm sine}=\ind_{[-1/2,1/2]}$. Hence
\[
    S_\eta(\xi)
    =1-\widehat{|K_{\rm sine}|^2}(\xi)
    =1-(\ind_{[-1/2,1/2]}*\ind_{[-1/2,1/2]})(\xi)
    =|\xi|,
    \qquad |\xi|\leq1.
\]
Thus, for any $r_0\leq1$, the model holds exactly with
$A_r=r/2$, $s=0$, $\alpha=1$, $c=1/2$ and zero remainder. By
\cite[Theorem~1.1]{Osa21}, its translation system is isomorphic to that of a
homogeneous Poisson process, hence is essentially free and CPE; therefore
Theorem~\ref{thm:cpe-supplies-F} applies. Also $S_\eta\in L^2(B_r)$.
\end{example}

\begin{example}[Thinned sine process]
For independent $p$-thinning, $0<p<1$, the pair-correlation function is
unchanged and the intensity is multiplied by $p$, so
\[
    S_{\eta,p}=1-p+pS_\eta,
    \qquad
    A_{p,r}=1-p+\frac p2r.
\]
Thus, for any $r_0\leq1$, the model holds exactly with
$s=1-p>0$, $\alpha=1$, $c=p/2$ and zero remainder. The process is determinantal with multiplier
$p\ind_{[-1/2,1/2]}$, which is integrable and takes values in $[0,1]$;
\cite[Theorem~1.1]{Osa21} makes its translation system isomorphic to a
homogeneous Poisson system, hence essentially free and CPE.
\end{example}

\begin{example}[Gaussian determinantal process]
This non-hyperuniform example has a nonzero remainder and makes the role of
$L$ explicit. Let $0<p\leq1$ and take Fourier multiplier
$\varphi(\xi)=p e^{-\pi\|\xi\|^2}$. It is integrable with values in $[0,1]$,
so it defines a determinantal process of intensity $\rho=p$; by
\cite[Theorem~1.1]{Osa21}, its translation system is isomorphic to a
homogeneous Poisson system. Writing $a:=p2^{-d/2}$,
\[
    S_\eta(\xi)=1-\frac{(\varphi*\varphi)(\xi)}{\rho}
    =1-ae^{-\pi\|\xi\|^2/2}.
\]
Since $|e^{-x}-(1-x)|\leq x^2/2$ for $x\geq0$,
\[
    \left|S_\eta(\xi)-(1-a)-\frac{a\pi}{2}\|\xi\|^2\right|
    \leq \frac{a\pi^2}{8}\|\xi\|^4.
\]
Proposition~\ref{prop:pointwise-to-integrated} gives
\[
    s=1-a>0,\qquad \alpha=2,\qquad \beta=4,\qquad
    c=\frac{a\pi d}{2(d+2)},\qquad
    L=\frac{a\pi^2 d}{8(d+4)}.
\]
Its translation action is therefore essentially free and CPE, so
Theorem~\ref{thm:cpe-supplies-F} applies.
\end{example}

\subsection{Inference from the limiting variables}
\label{subsec:intercept-inference}

Let $Y^{(\infty)}$ be the array from Corollary~\ref{cor:two-radius-limit}. For
$a\in\{r,\kappa r\}$,
\[
    \E Y^{(\infty)}_{a,j}=A_a,
    \qquad
    \tau_a^2:=\Var(Y^{(\infty)}_{a,j})
    =\frac{2}{\lambda_d(B_a)}
       \int_{B_a}S_\eta(\xi)^2\,d\lambda_d(\xi)-A_a^2,
\]
because $\E W_{a,j}=1$ and $\E W_{a,j}^2=2$. Put
\[
    \nu_r^2
    :=\frac{\kappa^{2\alpha}\tau_r^2+\tau_{\kappa r}^2}
            {(\kappa^\alpha-1)^2}.
\]
For an arbitrary data array $y$, define
\[
    \operatorname{CI}_r(y;\gamma)
    :=\left[
      \widehat s_r(y)-b_r-z_{1-\gamma/2}\operatorname{se}_r(y),
      \widehat s_r(y)+b_r+z_{1-\gamma/2}\operatorname{se}_r(y)
    \right]\cap[0,\infty).
\]

\begin{lemma}[Degenerate case]
\label{lem:degenerate}
Assume \eqref{eq:intro-integrated-model}, with $L>0$ and
$0<r\leq r_0/\kappa$. If $\nu_r^2=0$, then $S_\eta=0$ almost everywhere on $B_{\kappa r}$,
$s=s_r=0$, and every entry of $Y^{(\infty)}$ vanishes almost surely. Hence the
limiting rule never rejects and
$\operatorname{CI}_r(Y^{(\infty)};\gamma)=[0,b_r]$. Moreover, for every fixed
$m$,
\[
    \Pp\bigl(\operatorname{Rej}_r(Y^{(R)};\gamma)\bigr)\longrightarrow0,
    \qquad
    \Pp\bigl(s\in\operatorname{CI}_r(Y^{(R)};\gamma)\bigr)\longrightarrow1
\]
as $R\to\infty$.
\end{lemma}

\begin{proof}
Jensen's inequality gives $\tau_a^2\geq A_a^2$. Thus $\nu_r^2=0$ forces
$A_{\kappa r}=0$, and $S_\eta\geq0$ gives
$S_\eta=0$ almost everywhere on $B_{\kappa r}$. Hence $A_t=0$ for
$t\leq\kappa r$, so $s=s_r=0$ and the limiting array is zero. At fixed $m$,
$Y^{(R)}\dto0$ implies $Y^{(R)}\to0$ in probability. The continuous function
\[
    g(y):=\widehat s_r(y)-b_r-z_{1-\gamma}\operatorname{se}_r(y)
\]
satisfies $g(0)=-b_r<0$ because $L>0$, so the rejection probability tends to
zero. For coverage, the untruncated lower and upper endpoints are continuous
and equal $-b_r$ and $b_r$ at $y=0$; hence they straddle $s=0$ throughout a
neighborhood of the zero array, and the coverage probability tends to one.
\end{proof}

\begin{theorem}[Central limit theorem for the two-radius estimator]
\label{thm:two-scale-inference}
If $\int_{B_{\kappa r}}S_\eta^2\,d\lambda_d<\infty$ and $\nu_r^2>0$, then
\[
    \frac{\widehat s_r(Y^{(\infty)})-s_r}
         {\operatorname{se}_r(Y^{(\infty)})}
    \dto N(0,1)
    \qquad (m\to\infty).
\]
\end{theorem}

\begin{proof}
Apply the bivariate i.i.d. central limit theorem to
$(Y^{(\infty)}_{r,j},Y^{(\infty)}_{\kappa r,j})$. Its covariance matrix is
diagonal with entries $\tau_r^2$ and $\tau_{\kappa r}^2$. The linear map
$(x,y)\mapsto(\kappa^\alpha x-y)/(\kappa^\alpha-1)$ gives the numerator
limit, while the two sample variances in \eqref{eq:generic-two-radius} are
consistent. Slutsky's theorem completes the proof.
\end{proof}

\begin{corollary}[Fixed-radius inference]
\label{cor:fixed-radius-inference}
Assume $\int_{B_{\kappa r}}S_\eta^2\,d\lambda_d<\infty$ and $\nu_r^2>0$.
Then
\[
    \liminf_{m\to\infty}
    \Pp\!\left(
      \left|\widehat s_r(Y^{(\infty)})-s\right|
      < b_r+z_{1-\gamma/2}\operatorname{se}_r(Y^{(\infty)})
    \right)\geq1-\gamma.
\]
In particular,
$\liminf_{m\to\infty}\Pp(s\in\operatorname{CI}_r(Y^{(\infty)};\gamma))
\geq1-\gamma$.
If $s_r\leq b_r$, then
\[
    \liminf_{m\to\infty}
    \Pp\bigl(\widehat s_r(Y^{(\infty)})
      < b_r+z_{1-\gamma}\operatorname{se}_r(Y^{(\infty)})\bigr)
    \geq1-\gamma,
\]
whereas $s_r>b_r$ implies
\[
    \Pp\bigl(\operatorname{Rej}_r(Y^{(\infty)};\gamma)\bigr)\longrightarrow1.
\]
In particular, $H_0$ is contained in the first case, and $s>2b_r$ is
sufficient for power tending to one.
\end{corollary}

\begin{proof}
Theorem~\ref{thm:two-scale-inference} gives
$\Pp(|\widehat s_r-s_r|<z_{1-\gamma/2}\operatorname{se}_r)\to1-\gamma$.
Together with $|s_r-s|\leq b_r$, this gives the strict coverage bound and hence
the confidence-interval bound. If $s_r<b_r$, the fixed gap $b_r-s_r$ makes the strict nonrejection
probability tend to one; at $s_r=b_r$, the central limit theorem gives the
limit $1-\gamma$. If $s_r>b_r$, consistency of $\widehat s_r$ and
$\operatorname{se}_r\to0$ gives power tending to one. Finally,
$s_r\geq s-b_r$, so $s>2b_r$ implies $s_r>b_r$.
\end{proof}

\begin{remark}[Sampling error]
The two-radius cancellation increases sampling error. If $S_\eta$ is nearly
constant with value $s>0$ on $B_{\kappa r}$, then
$\tau_r^2\approx\tau_{\kappa r}^2\approx s^2$ and
\[
    \operatorname{se}_r(Y^{(\infty)})
    \approx \frac{s}{\sqrt m}
      \frac{\sqrt{\kappa^{2\alpha}+1}}{\kappa^\alpha-1}.
\]
Decreasing $r$ improves the bias bound $b_r$ but does not by itself reduce the
relative sampling error.
\end{remark}

\begin{remark}[Choice of $\kappa$]
The bias factor $D_{\alpha,\beta}(\kappa)$ diverges as $\kappa\downarrow1$ and
as $\kappa\to\infty$. For $(\alpha,\beta)=(1,2)$ it is minimized at
$\kappa=1+\sqrt2$, where $D_{1,2}=3+2\sqrt2$. Since the sampling variance also
depends on the unknown $S_\eta$, no choice of $\kappa$ is universally optimal.
\end{remark}

\begin{remark}[Misspecified exponent]
If the rule uses $\widetilde\alpha\neq\alpha$, the leading term is not
cancelled exactly: a residual of order $r^\alpha$ remains, so the $r^\beta$
bias bound and the stated level guarantee no longer apply.
\end{remark}

\begin{proof}[Proof of Theorem~\ref{thm:main}]
Fix $r$ and $m$. Corollary~\ref{cor:two-radius-limit} gives
$Y^{(R)}\dto Y^{(\infty)}$. If $\nu_r^2=0$, Lemma~\ref{lem:degenerate} gives
both fixed-radius claims. Suppose $\nu_r^2>0$. The sets
\[
    \{\widehat s_r<b_r+z_{1-\gamma}\operatorname{se}_r\},
    \qquad
    \{\widehat s_r>b_r+z_{1-\gamma}\operatorname{se}_r\}
\]
are open, so Portmanteau bounds their probabilities for $Y^{(R)}$ from below
by those for $Y^{(\infty)}$ as $R\to\infty$. Corollary~\ref{cor:fixed-radius-inference}
then gives the fixed-radius size and power statements. If $r_n\downarrow0$,
these bounds hold for every $n$; under a fixed $s>0$, one has
$b_{r_n}\to0$ and hence $s>2b_{r_n}$ eventually. This proves the two final
claims.
\end{proof}

\begin{corollary}[Confidence interval from the observed pattern]
\label{cor:observed-ci}
For every fixed $0<r\leq r_0/\kappa$,
\[
    \liminf_{m\to\infty}\liminf_{R\to\infty}
    \Pp\bigl(s\in\operatorname{CI}_r(Y^{(R)};\gamma)\bigr)\geq1-\gamma.
\]
\end{corollary}

\begin{proof}
For $\nu_r^2=0$, use Lemma~\ref{lem:degenerate}. Otherwise the event
\[
    \left\{\left|\widehat s_r-s\right|
      < b_r+z_{1-\gamma/2}\operatorname{se}_r\right\}
\]
is open and implies $s\in\operatorname{CI}_r$. Portmanteau at fixed $m$ and
the strict coverage bound in Corollary~\ref{cor:fixed-radius-inference} give
the claim.
\end{proof}

\section{Variants and diagnostics}
\label{sec:variants}

\subsection{Sensitivity to the remainder bound}
\label{subsec:L-sensitivity}

Since $b_r=L D_{\alpha,\beta}(\kappa)r^\beta$, the rule rejects exactly when
\[
    L<L_{\max}(y)
    :=\frac{[\widehat s_r(y)-z_{1-\gamma}\operatorname{se}_r(y)]_+}
            {D_{\alpha,\beta}(\kappa)r^\beta},
    \qquad [x]_+:=\max\{x,0\}.
\]
We report $L_{\max}$ as a sensitivity summary: it shows how large the remainder
constant can be before the decision changes. It does not estimate $L$, and
taking $L$ below the true remainder constant voids the level guarantee.

\subsection{A direct periodogram for CPE processes}
\label{subsec:ball-implementation}

For the CPE processes of Theorem~\ref{thm:cpe-supplies-F},
\cite[Theorem~6.1 and Corollaries~6.2--6.3]{Bjo26} also gives the Fourier
limit for the unweighted sum over $B_R$. The raw periodogram
(or scattering intensity)
\[
    I_R(u):=\frac1{\eta(B_R)}
      \left|\sum_{x\in\eta\cap B_R}\chi_u(x)\right|^2,
\]
with $I_R(u)=0$ when $\eta(B_R)=0$, therefore replaces $I_{R,h}$. For every
fixed finite frequency tuple in a set of full Lebesgue measure,
\[
    (I_R(u_1),\ldots,I_R(u_m))
    \dto
    (S_\eta(u_1)W_1,\ldots,S_\eta(u_m)W_m).
\]
Conditioning on absolutely continuous random frequencies gives the analogue of
Theorem~\ref{thm:observed-randomized-limit}, so the two-radius rule applies with
$Y^{(R)}_{a,j}:=I_R(U_{a,j})$. No auxiliary test function or separate intensity
estimate is needed.

\section{A design remark}
\label{sec:design}

Assume the stronger finite-family Fourier limit supplied by
\cite[Theorem~5.1 and Corollary~5.5]{Bjo26} under the hypotheses of
Theorem~\ref{thm:cpe-supplies-F}. With a fixed budget of Fourier coefficients, one test function per sampled frequency minimizes the variance in
the limiting model. To see this, let $U$ be uniform on $B_r$ and put
\[
    M_{2,r}:=\E[S_\eta(U)^2],
    \qquad
    V_r:=\Var(S_\eta(U)).
\]
Let $h_1,\ldots,h_K\in C_c^\infty(B_1;\bC)$ be orthonormal in
$L^2(\lambda_d)$. Properness and orthogonality make their Gaussian limits at a
fixed frequency independent. Averaging the squared limits replaces $W$ by
$\overline W_K=K^{-1}\sum_{\ell=1}^K W_\ell$, where
$\E\overline W_K=1$ and $\E\overline W_K^2=1+K^{-1}$. If $m$ frequencies are
sampled and $J=mK$ is the total number of Fourier coefficients, the sample mean
of $S_\eta(U_j)\overline W_{K,j}$ has variance
\[
    \frac{K V_r+M_{2,r}}{J}.
\]
Hence $K=1$ minimizes the variance when $V_r>0$, while all allocations agree
when $V_r=0$. Nonuniform frequency sampling can also change the variance, but
its oracle choice depends on the unknown structure factor and is not pursued
here.

\section{Conclusion}
\label{sec:conclusion}

Randomizing frequencies turns the almost-everywhere Fourier limit into
inference for the intrinsic Bartlett averages $A_r$. The two-radius
extrapolation removes the prescribed leading term and gives an explicit bias
allowance, yielding confidence intervals and a test of asymptotic level at
most $\gamma$ from one observed point pattern. Joint regimes $r=r(R)$ or $m=m(R)$,
finite-$R$ calibration, and estimation or validation of the low-frequency
parameters require additional input.

\subsection*{Acknowledgments}
The author thanks G\"unter Last for encouraging discussions on his paper with
Henze and Klatt.

\section*{Statements and Declarations}

\medskip\noindent\textbf{Funding.}\ 
This work was supported by the Swedish Research Council under grant
VR 11253322.

\medskip\noindent\textbf{Competing interests.}\ 
The author has no relevant financial or non-financial interests to disclose.

\medskip\noindent\textbf{Data availability.}\ 
No datasets were generated or analysed during the current study.

\medskip\noindent\textbf{Use of artificial-intelligence tools.}\ 
The author used ChatGPT (OpenAI) for editorial assistance, consistency checks,
and LaTeX maintenance, and Claude (Anthropic) and Aristotle (Harmonic) for
independent manuscript and proof audits. Aristotle was also used for isolated
algebraic and statistical checks. All outputs were independently evaluated by the
author, who takes full responsibility for the mathematical content, references,
and final text.

\end{document}